\documentclass[12pt]{amsart}

\usepackage{fullpage}
\usepackage{amsmath}
\usepackage{amssymb}
\usepackage{amsthm}
\usepackage{amstext}
\usepackage{appendix}
\usepackage{perpage}
\usepackage{slashed}
\usepackage[x11names]{xcolor}
\usepackage{tikz-cd}
\usepackage{ytableau}\ytableausetup{centertableaux}
\usepackage[colorlinks=true, allcolors=blue]{hyperref}
\usepackage{todonotes}

\newtheorem{thm}{Theorem}[section]
\newtheorem{lemm}[thm]{Lemma}

\newtheorem{prop}[thm]{Proposition}

\theoremstyle{definition}

\newtheorem{defi}[thm]{Definition}
\newtheorem{remark}[thm]{Remark}

\begin{document}

\title{Revisiting the infinite  Galois correspondence via group representations}

\author{Zhe Chen}

\author{Enxin Wu}

\address{Department of Mathematics, Shantou University, Shantou, 515821, China}
\email{zhechencz@gmail.com}

\address{Department of Mathematics, Shantou University, Shantou, 515821, China}
\email{exwu@stu.edu.cn}

\begin{abstract}
Inspired by Lenstra's infinite normal basis theorem, we interpret the  Galois correspondence as a special case of a simple general principle in harmonic analysis on profinite groups. This also gives a representation-theoretic characterisation of profinite groups.
\end{abstract}

\maketitle

\section{Introduction}

Galois theory is fundamental in algebra, and is one of the most fascinating subjects often taught in undergraduate maths programmes. Its main theorem could be stated as: 

\begin{thm}\label{thm:classical}
Let $L/K$ be a (possibly infinite) Galois extension of fields with (profinite) Galois group $G$, and let $S_G:=\{ \textrm{closed subgroups of}\ G\}$. Then the map
$$H\longmapsto L^H$$
from $S_G$ to the set of intermediate fields in $L/K$ satisfies the following properties:
\begin{itemize}
\item[(i)] This map is an inclusion-reversing bijection.
\item[(ii)] If $H \subseteq G$ is a closed subgroup of finite index, then $\dim_K L^H=[G:H]$.
\item[(iii)] If $H \subseteq G$ is a normal closed subgroup, then $L^H/K$ is Galois with Galois group $G/H$.
\end{itemize}
\end{thm}

When $L/K$ is a finite extension with automorphism group $G$, thanks to the celebrated normal basis theorem (for its history, see  \cite[Page-71~Notes]{milne_Fields_Galois_2022}), $L/K$ is Galois if and only if, as $K$-representations of $G$, one has 
$$L\cong K[G].$$
In 1985 Lenstra generalised the normal basis theorem to infinite extensions in \cite[Theorem~1]{Lenstra_normal_basis_thm_1985} (see also \cite{Lundstrom_cohom_normal_basis_thm_2007} for a different proof), by using the space of continuous functions $C(G,K)$ instead of $K[G]$. This inspires us, by forgetting the field extension $L/K$ and  focusing solely on the group $G$, to interpret Theorem~\ref{thm:classical} as a special case of a simple general property of regular representations of profinite groups.

\vspace{2mm} Recall that for a finite subgroup $H\subseteq G$, a theorem of Artin (see e.g.\ \cite[Corollary~3.11]{milne_Fields_Galois_2022}) asserts that $L/L^H$ is  Galois. This suggests  the following notion.

\begin{defi}\label{defi: submodules}
For $G$ a topological group and $R$ a  non-zero ring (not necessarily commutative or unital), let $C(G,R)$ be the left $R$-module of continuous maps from $G$ to $R$ where $R$ is equipped with the discrete topology. Endow $C(G,R)$ with the $G$-action given by $g\cdot f(-)=f(g^{-1}(-))$. The \emph{Artin lattice} of the pair $(G,R)$ is
$$\mathbb{A}(G,R)=\{ C(G,R)^H \mid H\ \textrm{is a closed subgroup of}\ G \},$$
in which $C(G,R)^H:=\{ f \in C(G,R) \mid h \cdot f(-) = f(-),\   \forall h \in H \}$.
\end{defi}

Artin lattice leads to the following exact generalisation of Theorem~\ref{thm:classical}:

\begin{thm}\label{thm: main}
Let  $G$ be a profinite group and $R$ a  non-zero ring. Then the map
$$\mathcal{F}\colon S_{G} \longrightarrow  \mathbb{A}(G,R)$$
given by 
$$H\longmapsto C(G,R)^H$$
satisfies the following properties:
\begin{itemize}
\item[(i)] The map $\mathcal{F}$ is an inclusion-reversing bijection, whose inverse is the (pointwise stabiliser)  map  
$$Z_G(-)\colon A\longmapsto  \{ g \in G \mid g \cdot a = a,\ \forall a \in A\}.$$
\item[(ii)] If $H\in S_G$ is of finite index, then $\mathcal{F}(H)\cong R^{\# G/H}$ as left $R$-modules.
\item[(iii)] If $H \in S_G$ is normal, then the embedding of $C(G/H,R)$ into $C(G,R)$ (via pre-composition with the quotient map $G \to G/H$) is an isomorphism onto $\mathcal{F}(H)=C(G,R)^H$  as   left $R$-modules of $G/H$.
\end{itemize}
\end{thm}

\begin{remark}
While a profinite group is compact, the above theorem does not hold for general compact groups: For example, for the  circle group $S^1$  we have 
$$\mathbb{A}(S^1,\mathbb{F}_2)=\{ \mathbb{F}_2 \};$$
so Theorem~\ref{thm: main}(i) fails. Actually, our proof of Theorem~\ref{thm: main} is topological and relies heavily on the profinite topology on $G$.
\end{remark}

There is  a partial converse to Theorem~\ref{thm: main}, which serves as a representation-theoretic characterisation of profinite groups:

\begin{thm}\label{thm: charac}
Let $G$ be a compact Hausdorff  group. Then $G$ is profinite if and only if, for one (or equivalently, every) non-zero ring $R$, the map $\mathcal{F}\colon S_G\rightarrow \mathbb{A}(G,R)$  in Theorem~\ref{thm: main} is a bijection.
\end{thm}

In Theorem~\ref{thm: main}, if one replaces the ring $R$ by a general discrete set consisting of at least two elements, then the proof in Section~\ref{sec:proof of main} implies that the same conclusion holds true by forgetting the module structure on $C(G,R)^H$.

\vspace{2mm} Note that even in the case $R=K$ is a field, Theorem~\ref{thm: main} strictly generalises the situation of Theorem~\ref{thm:classical} because it allows $K$ to be any field, including those admitting no non-trivial algebraic extensions.  On the other hand, it is not difficult to see that when $R=K$ is a field and $C(G,K)$ admits a field extension structure with the $G$-action being the field automorphism, the multiplication of the field is in general neither  the usual function multiplication nor the convolution multiplication. Moreover, note that Theorem~\ref{thm: main} is not a special case of Grothendieck's schematic Galois correspondence (see e.g.\ \cite[Theorem~5.4.2]{szamuely2009galois}) or of the Hopf--Galois correspondence (see e.g.\ Montgomery \cite[Section~6,7]{montgomery_2009_hopf}), because our base ring $R$ need not be commutative (nor even unital); we only need it to be non-zero.

\vspace{2mm} \noindent {\bf Acknowledgement.} We thank Stasinski for helpful comments on an earlier draft.

\section{Proofs of Theorem~\ref{thm: main} and of Theorem~\ref{thm: charac}}\label{sec:proof of main}

Our method uses only the following basic topological properties of profinite groups, which (and much more) can be found in  standard texts like \cite{Fried_Jarden_field_arith_2008_3rdedition} and \cite{Ribes_Zalesskii_profinite_2010_2nd}:

\vspace{2mm} If $G$ is a profinite group, then
\begin{itemize}
\item $G$ is Hausdorff;
\item $G$ is compact (in particular, its open subgroups are precisely the closed subgroups of finite index);
\item $G$ is totally disconnected;
\item the open normal subgroups of $G$ form a neighbourhood basis of the identity  $1$;
\item closed subgroups of $G$ are profinite;
\item the quotient of $G$ by a closed normal subgroup is  profinite.
\end{itemize}

\begin{lemm}\label{lemm-main}
Suppose that $G$ is profinite. Then for any $H \in S_G$ one has
\begin{equation*}
\mathcal{F}(H)=\sum_{N} \mathrm{Map}(HN\backslash G, R)=\bigcup_{N} \mathrm{Map}(HN\backslash G, R),
\end{equation*}
where $N$ runs over all open normal subgroups of $G$ and $\mathrm{Map}(HN\backslash G, R)$ means the  set of   maps from the set of right cosets $HN\backslash G$ to $R$. Here note that $HN$ is a closed (and open) subgroup
, and $\mathrm{Map}(HN\backslash G, R)$ is an $R$-submodule of $C(G,R)$ via the pre-composition with the quotient map $G \to HN\backslash G$.
\end{lemm}

\begin{proof}
Let $f\in C(G,R)$. As each point in $R$ is open, there is a  partition of $G$ into open subsets such that $f$ is constant on each part. Since $G$ is profinite, the cosets of open normal subgroups form a topological basis, and therefore each part of the partition is a union of cosets. Since $G$ is compact, there exists a finite family of, and hence a single, open normal subgroup  $N_f$ such that each part of the partition is a union of some cosets of $N_f$. In particular, $f$ belongs to $\mathrm{Map}(N_f\backslash G, R)$. 

\vspace{2mm} If furthermore $H$ stabilises $f$, then $f$ would be constant on each coset of $HN_f$, and hence $f\in \mathrm{Map}(HN_f\backslash G, R)$. This implies that
$$\mathcal{F}(H)=C(G,R)^H=\sum_{N} \mathrm{Map}(HN\backslash G, R)$$
where $N$ runs over all open normal subgroups of $G$. 
\end{proof}

We shall  show that the pointwise stabiliser  map $Z_G(-)$ is well-defined, i.e.\ its images are closed; we prove the following stronger property:

\begin{prop}
Suppose that $G$ is profinite. Then for any subset $S\subseteq C(G,R)$, the pointwise stabiliser subgroup $Z_G(S)\subseteq G$ is closed.
\end{prop}
\begin{proof}
Let $s\in C(G,R)$. By Lemma~\ref{lemm-main} there is an open normal subgroup $N_s$ of $G$ stabilising $s$. So for any $\sigma\in G\setminus Z_G(s)$  we have 
$$\sigma N_s\cap Z_G(s)=\emptyset,$$
namely $\sigma$ lies in an open neighbourhood outside $Z_G(s)$, which implies the closedness of $Z_G(s)$. Therefore $Z_G(S)=\bigcap_{s\in S}Z_G(s)$ is closed as well.
\end{proof}

Now we turn to the proof of Theorem~\ref{thm: main}.

\begin{proof}[Proof of Theorem~\ref{thm: main}]
(i) By definition $\mathcal{F}$ is surjective. So, since $Z_G(-)$ reverses the inclusion relation, we only need to prove the injectivity. It is clear that
\begin{equation}\label{temp1}
H \subseteq Z_{G}(C(G,R)^{H});
\end{equation}
if we could show that \eqref{temp1} is actually an equality then the assertion follows.

\vspace{2mm} Let  $x\in Z_{G}(C(G,R)^{H})$.  If $x$ is not contained in $H$, by the Hausdorff property there would be an open subset $V\subseteq G$ such that   $H\subseteq V$ and $x\notin V$. 

\vspace{2mm} Consider the multiplication map
$$m\colon G\times G\longrightarrow G;$$
by construction  $m^{-1}(V)$ is an open subset containing $H\times \{1\}$, so there are open subsets $U_i,U_i'\subseteq G$ ($i\in I$) satisfying that: $H\subseteq \bigcup_{i\in I} U_i$, $1\in  U'_i$ ($\forall i\in I$), and $\bigcup_{i\in I} (U_i\times U_i')\subseteq m^{-1}(V)$. Since $H$ is compact, we can find a finite subset $J$ of the index set $I$  such that  $H\subseteq \bigcup_{i\in J} U_i$. In particular, if we let $U':=\bigcap_{i\in J}U'_i$ (which contains $1$), then 
$$ H\times \{1\}\subseteq \left(\bigcup_{i\in J} U_i\right)\times U'=\bigcup_{i\in J} \left(U_i\times U'\right)\subseteq \bigcup_{i\in J} \left(U_i\times U_i'\right)\subseteq m^{-1}(V).$$
Since $U'$ is open (as $J$ is finite) and contains $1$, there is an open normal subgroup $N_V\subseteq G$ contained in $U'$, and thus
$$H\times N_V\subseteq  \left(\bigcup_{i\in J} U_i\right)\times U'\subseteq m^{-1}(V).$$
This implies that $HN_V\subseteq V$. 

\vspace{2mm} By Lemma~\ref{lemm-main} and the assumption that $R$ is non-zero, for every open normal subgroup $N\vartriangleleft G$ we have   $xHN=HN$, and hence $x\in HN$. So in particular 
$x\in HN_V$, which implies that $x\in V$, a contradiction. 

\vspace{2mm} Thus $x\in H$ and therefore \eqref{temp1} is an equality.

\vspace{2mm}(ii)
If $H$ is of finite index in $G$, then $H$ is open, and hence by Lemma~\ref{lemm-main} we get
$$\mathcal{F}(H)=C(G,R)^H\cong\mathrm{Map}(H\backslash G, R).$$
So $\mathcal{F}(H)\cong R^{\# G/H}$.

\vspace{2mm}
(iii)
Since $H$ is normal, the open normal subgroups of $G/H$  are in one-to-one correspondence (by taking pre-image) with  the open normal subgroups of $G$ containing $H$; the latter groups are precisely those of the form $HN$ for $N$ being open normal in $G$. So the assertion follows immediately from Lemma~\ref{lemm-main}.
\end{proof}

\begin{proof}[Proof of Theorem~\ref{thm: charac}]
It suffices to prove the ``if'' part, and for this it suffices to prove that $G$ is totally-disconnected (see e.g.\ \cite[Theorem~1.1.12]{Ribes_Zalesskii_profinite_2010_2nd}).

\vspace{2mm} Let $G^0$ be the connected component of $G$ containing the identity $1$; this is a closed connected subgroup of $G$. Given $f\in C(G,R)$, for any $x\in G$ and $g\in G^0$ one has
$$g\cdot f(x)=f(g^{-1}x)=f(G^0x)=f(x),$$
in which the second and the third equalities follow from that $f$ takes a connected set to a connected set. So
$$C(G,R)=C(G,R)^{G^0}.$$
Since obviously $C(G,R)=C(G,R)^{\{1\}}$, we get $C(G,R)^{G^0}=C(G,R)^{\{1\}}$. As $\{1\}$ is a closed subgroup of $G$ by the Hausdorff property, the injectivity of  $\mathcal{F}\colon S_{G} \rightarrow  \mathbb{A}(G,R)$ implies that $G^0=\{1\}$. So $G$ is totally-disconnected.
\end{proof}

\bibliographystyle{alpha}
\bibliography{zchenrefs}

\end{document}